\documentclass[12pt, a4paper, reqno]{amsart}
\usepackage[russian,english]{babel}
\usepackage{amsmath,amsfonts,amssymb,amsthm}
\usepackage{epsfig}
\usepackage{pgf}
\usepackage{hyperref}
\usepackage{graphicx,graphics}
\usepackage{epstopdf}
\usepackage{color}
\usepackage{caption}
\usepackage{subcaption}
\usepackage{longtable}

\usepackage[margin=2.5cm]{geometry}

\newtheorem{thm}{Theorem}[section]

\newtheorem{con}[thm]{Conjecture}
\newtheorem{prop}[thm]{Proposition}
\theoremstyle{definition}
\newtheorem{defin}[thm]{Definition}
\newtheorem{exam}[thm]{Example}
\theoremstyle{remark}
\newtheorem*{rem}{Remark}

\newcommand { \ib }[1] {\textit{\textbf{#1}}}

\newcommand{\R}{\mathbb{R}}
\newcommand{\Q}{\mathbb{Q}}
\newcommand{\Z}{\mathbb{Z}}

\begin{document}
\renewcommand{\ib}{\mathbf}
\renewcommand{\proofname}{Proof}
\renewcommand{\phi}{\varphi}
\newcommand{\conv}{\mathrm{conv}}

\title[]{On $\pi$-surfaces of four-dimensional parallelohedra}
\author{Alexey~Garber$^{*}$}

\address{Faculty of Mechanics and Mathematics, Moscow State University, Russia, 119991, Moscow, Leninskie gory, 1, and Delone Laboratory of Discrete and Computational Geometry, Yaroslavl State University, 14 Sovetskaya str., Yaroslavl, 150000, Russia.}
\email{alexeygarber@gmail.com}

\thanks{{$^{*}$}Supported by the Russian Foundation of Basic Research grants 11-01-00633-a and 11-01-00735-a, and the Russian government project 11.G34.31.0053.}

\date{\today}

\begin{abstract}
We show that every four-dimensional parallelohedron $P$ satisfies a recently found condition of Garber, Gavrilyuk \& Magazinov sufficient for the Voronoi conjecture being true for $P$. Namely we show that for every four-dimensional parallelohedron $P$ the group of rational first homologies of its $\pi$-surface is generated by half-belt cycles.
\end{abstract}

\maketitle

\section{The Voronoi conjecture and $\pi$-surface of parallelohedron}

In this paper we study one combinatorial property of four-dimensional parallelohedra that is closely related to the Voronoi conjecture. 

\begin{defin}
A convex polytope $P\subset\R^d$ is called a {\it parallelohedron} if we can tile $\R^d$ with translations of $P$.
\end{defin}

We will consider only {\it face-to-face} tilings by parallelohedra when intersection of any two copies of $P$ is face of both, possible empty. Some parallelohedra allows {\it non face-to-face} tilings too but as it is in the usual brickworks. But as it was shown by Venkov \cite{Ven} and McMullen \cite{McM} every parallelohedron $P$ has a correspondent (unique) face-to-face tiling $\mathcal{T}(P)$. So for now we will consider only face-to-face tilings.

The Voronoi conjecture establishes a direct connection between $d$-dimensional parallelohedra and Dirichlet-Voronoi domains of $d$-dimensional lattices.

\begin{defin}\label{vortiling}
Given a lattice $\Lambda\subset\R^d$. The {\it Dirichlet-Voronoi domain} of the lattice $\Lambda$ is the polytope consist of points of $\R^d$ that are closer to a fixed point $O\in \Lambda$ than to any other point of $\Lambda$.
\end{defin}

It is clear that Dirichlet-Voronoi domain of an arbitrary $d$-dimensional lattice is a $d$-dimensional parallelohedron. The Voronoi conjecture claims that we can describe all parallelohedra via Dirichlet-Voronoi domains.

\begin{con}[The Voronoi conjecture on parallelohedra, \cite{Vor}]\label{vorcon}
Every $d$-dimensional parallelohedron is an affine image of Dirichlet-Voronoi domain of some $d$-dimensional lattice.
\end{con}

The Voronoi conjecture was poved for several cases but still remains unproved in general. In this section we will describe some known results concerning this conjecture.

The first result was obtained by Voronoi himself in 1909 \cite{Vor}.

\begin{defin}
A $k$-face $F$ of a parallelohedron $P\subset \R^d$ is called {\it primitive} if it belongs to exactly $d+1-k$ copies of $P$ in the correspondent tiling $\mathcal{T}(P)$.

If all $k$-faces of $P$ are primitive then $P$ and the correspondent face-to-face tiling $\mathcal{T}(P)$ are called {\it $k$-primitive}.
\end{defin}

\begin{thm}[G.~Voronoi, \cite{Vor}]\label{vorthm}
The conjecture $\ref{vorcon}$ is true for $0$-primitive (or just primitive) parallelohedra.
\end{thm}

\begin{thm}[O.~Zhitomirskii, \cite{Zhit}]\label{zhitthm}
The Voronoi conjecture $\ref{vorcon}$ is true for $(d-2)$-primitive parallelohedra.
\end{thm}

It is easy to see that if a parallelohedron $P$ is $k$-primitive then it is $(k-1)$-primitive also so Zhitmorski theorem implies the theorem of Voronoi. Zhitomirski formulated his result not with notion of primitive face but using the notion of {\it belt} introduced by Minkowski.

\begin{thm}[H.~Minkowski, \cite{Min}]
Every $d$-dimensional parallelohedron satisfies the following conditions (so-called Minkowski conditions):
\begin{enumerate}
\item $P$ is centrally symmetric;
\item Every facet of $P$ is centrally symmetric;
\item Projection of $P$ along any its face of codimension $2$ is a parallelogram or centrally symmetric hexagon.
\end{enumerate}
\end{thm}

For arbitrary face $F$ of codimension $2$ we can take the set of all facets of $P$ parallel to $F$. This set of facet is called the {\it belt} $\mathcal{B}_F$. The second condition of Minkowski implies that the belt $\mathcal{B}_F$ is ``closed'', i.e. every facet from $\mathcal{B}_F$ contains exactly two faces parallel to $F$. The third condition of Minkowski implies that every belt of $P$ consist of $4$ or $6$ facets.

Among all faces of codimension $2$ the primitive are only those which generates $6$-belts, so the theorem of Zhitomirskii \ref{zhitthm} can be formulated as follows: if all belts of $P$ consist of $6$ facets then $P$ satisfies the Voronoi conjecture.

Further way of investigating combinatorial properties of parallelohedra and their connection with Voronoi conjecture was presented by Ordine \cite{Ord}. Ordine gave a condition that implies the Voronoi conjecture using local combinatorial structure of $(d-3)$-faces of the tiling $\mathcal{T}(P)$.

\begin{defin}
Consider an arbitrary face $F$ of codimension $k$ of the tiling $\mathcal{T}(P).$ Then the convex hull of centers of all tiles sharing $F$ forms a {\it dual $k$-cell} $\mathcal{D}_F$ corresponding to $F$. We will refer to $k$ as (combinatorial) dimension of $k$-cell. For example, the dual cell corresponding to a $d$-dimensional polytope $P'\in \mathcal{T}(P)$ is a single point --- its center.
\end{defin}

The set of all dual cells of the tiling $\mathcal{T}(P)$ determines a structure of combinatorial cell complex in $\R^d$. So, we can introduce straightforward definition of combinatorial equivalence of dual cells or dual cell and polytope.

\begin{defin}
$\mathcal{D}_F$ is said to be {\it combinatorially equivalent} to a polytope in case the polytope has identical face lattice. For example a dual $2$-cell $\mathcal{D}_F$ is equivalent either to a triangle for primitive $F$ or to a parallelogram for a non-primitive $F$ (see figure \ref{dual2cells}).
\end{defin}

\begin{figure}[!ht]
\begin{center}
\includegraphics[scale=0.9]{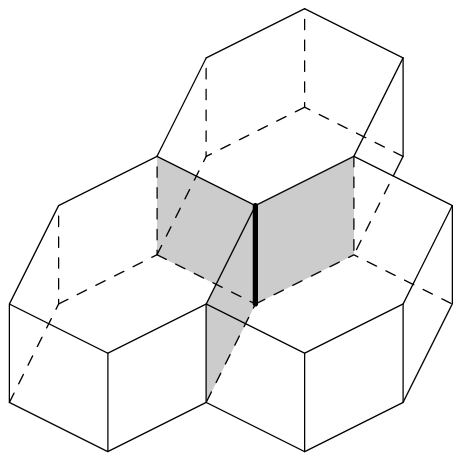}
\hskip 3cm
\includegraphics[scale=0.9]{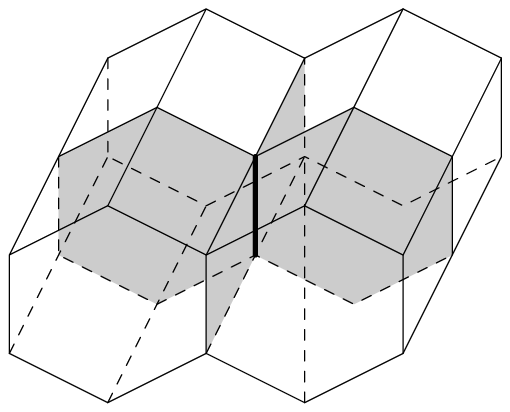}
\end{center}
\caption{Local structure of primitive and non-primitive $(d-2)$-faces.}
\label{dual2cells}
\end{figure}

Delone proved \cite{Del} (see also \cite{Ord}) that there are five possible combinatorial type of dual 3-cells. Using that result Delone proved the Voronoi conjecture \ref{vorcon} for four-dimensional parallelohedra and gave the list of 51 of them \cite{Del}. The last one missing 52nd parallelohedron was added by Shtogrin \cite{Sto}.

\begin{thm}[B.~Delone, \cite{Del}]\label{thmdel}
Every dual $3$-cell is combintorially equivalent to one of the five $3$-dimensional polytopes: cube, triangular prism, tetrahedron, octahedron, or quadrangular pyramid.
\end{thm}

\begin{thm}[A.~Ordine, \cite{Ord}]\label{ordthm}
If all dual $3$-cells of the tiling $\mathcal{T}(P)$ are not combinatorially equivalent to cube or triangular prism then $P$ satisfies the Voronoi conjecture.
\end{thm}

Also there is a recent result by Garber, Gavrilyuk and Magazinov \cite[Thm. 4.6]{GGM} that has a direct connection with this paper.

\begin{defin}\label{deltasur}
Consider the surface $\partial P$ of a $d$-dimensional parallelohedron $P$. After deletion of all closed non-primitive faces of codimension 2 of $P$ we  obtain a $(d-1)$-dimensional manifold without boundary. We will call this manifold the $\delta$-{\it surface} of $P$ and denote by $P_\delta$.

If we glue together every pair of opposite points of the $P_\delta$, we obtain another $(d-1)$-dimensional manifold that is a subset of real projective space $\mathbb{RP}^{d-1}$. We call this manifold the {\it $\pi$-surface} of $P$ and denote by $P_\pi$.
\end{defin}

\begin{defin}
A path on the $\delta$-surface of a parallelohedron $P$ is called {\it half-belt cycle} if it starts in the center of one facet $F$, ends in the center of the opposite facet $F'$ and passes only two other facets from one belt that contains $F$ and $F'$. It is clear that such a path represents a cycle on the $\pi$-surface of $P$.
\end{defin}

\begin{thm}[A.~Garber, A.~Gavrilyuk, and A.~Magazinov, \cite{GGM}]\label{homologies}
If the group of one-dimensional homologies $H_1(P_\pi,\Q)$ is generated by half-belts cycles of $P$ then the Voronoi conjecture $\ref{vorcon}$ is true for $P$.
\end{thm}

This results covers theorem of Voronoi \ref{vorthm} and Zhitomirski \ref{zhitthm} but for now it is unclear its connection with Ordine's theorem \ref{ordthm}. In the same paper \cite[Sect. 5]{GGM} it is shown that all three-dimensional parallelohedra satisfies theorem \ref{homologies}.

In this paper we will show that all four-dimensional parallelohedra satisfies this theorem, i.e. we will show that group of one-dimensional rational homologies of $\pi$-surface of an arbitrary four-dimensional parallelohedron is generated by its half-belt cycles. We need to remark that this will NOT give us a new proof of the Voronoi conjecture for four-dimensional case since we will use (see Section \ref{sec:class}) Delone classification \cite{Del} of four-dimensional parallelohedra which is based on his proof of the Voronoi conjecture for four-dimensional case.

\section{Classification of four-fimensional parallelohedra via Delone triangulations}\label{sec:class}

The first classification of four-dimensional parallelohedra appeared in the work of Delone \cite{Del}. Delone first proved the Voronoi conjecture for that class of polytopes and then used Schlegel diagrams \cite[Lect. 5]{Zieg} to give almost the full list later completed by Stogrin.

In this paper we will use another way of constructing the full list of four-dimensional parallelohedra using Voronoi decomposition of the cone of positive (semi)definite quadratic forms. This way also uses the fact that any four-dimensional parallelohedron satisfies the Voronoi conjecture and this way is described in details in \cite[Sect. 2--4]{Val}.

The set of all $d$-dimensional positive semidefinite quadratic forms is the convex cone that we will denote $\mathcal{C}_d$.

\begin{defin}
Fix a lattice $\Lambda\subset \R^d$. For an arbitrary positive definite $Q:\R^d\longrightarrow \R$ we can define the {\it Delone tiling $D_\Lambda(Q)$ of $\Lambda$ with respect to $Q$} by taking all lattice polytopes that are inscribed in ``empty'' ellipsoids with respect to $Q$. In other words, a $d$-dimensional polytope $P$ with vertex set $V_P\subset \Lambda$ is in $D_\Lambda(Q)$ iff there is a point $\mathbf c$ such that for any $\ib v\in V_P$ we have $Q(\ib v-\ib c)=\min_{\ib x\in \Lambda}Q(\ib x-\ib c)$ and no other lattice points attains this minimum.

We can generalize this definition to the case of semidefinite non-negative quadratic form and in that case we will obtain unbounded Delone tiles. We will consider only the case $\Lambda=\Z^d$ since any other lattice and Delone tiling can be transformed in this case with suitable affine transformation.
\end{defin}

Delone tilings are dual to Voronoi tilings of lattice with respect to different quadratic forms and the only difference from the definition \ref{vortiling} of the Voronoi domain $V_\Lambda(Q)$ with respect to quadratic form $Q$ is that we will consider metrics $Q(\ib x)$ instead of ``usual'' metrics $\rho(\ib x)=x_1^2+\ldots+x_d^2$. For example Delone tiling is a triangulation iff corespondent Voronoi polytope is primitive.

\begin{defin}
Let $\mathcal{D}$ be a Delone tiling of lattice $\Z^d$. The set $\triangle(\mathcal{D})$ of all positive quadratic forms  with the same Delone tiling is called {\it secondary cone} of the tiling $\mathcal D$. If $\mathcal{D}$ is the Delone tiling with respect to the form $Q$ then $\triangle(\mathcal{D})$ is also called the {\it $L$-type domain} of $Q$. It is a convex subcone of $\mathcal C_d$. 

Full-dimensional secondary cones correspondent to Delone triangulations of the lattice $\Z^d$ since if $Q$ determines a Delone triangulation then any sufficiently small perturbation of $Q$ does not change the correspondent triangulation.
\end{defin}

\begin{thm}[G.~Voronoi, \cite{Vor}]
The closures $\bigl\{\overline{\triangle(\mathcal{D})}\bigr\}$ of all secondary cones for Delone triangulations forms a face-to-face decomposition $\mathfrak{C}_d$ of the cone of positive semidefinite quadratic forms.
\end{thm}

It is clear that two proportional quadratic forms determines the same Delone tiling. So the decomposition $\mathfrak{C}_d$ of the cone $\mathcal{C}_d$ will not have vertices and the smallest possible face will have dimension $1$.

\begin{defin}
Forms that represents faces of the dimension $1$ of the tiling $\mathfrak{C}_d$ are called {\it rigid forms}. Any small perturbation of rigid form (except multiplying by positive real) changes correspondent Delone tiling.
\end{defin}

The following theorem is the fundamental theorem about how Delone tiling $D_\Lambda(Q)$ and Dirichlet-Voronoi polytope $V_\Lambda(Q)$ changes when we travel on the cone $\mathcal{C}_d$. It is proved in Ph.D. thesis of F.~Vallentin \cite[Prop. 2.6.1 and Prop. 3.3.5]{Val}, though its second part was announced in \cite{Rysh} and proved in \cite{BolR}.

\begin{thm}\label{onecone}
If all quadratic forms $Q_1,\ldots, Q_n$ contained in the closure of one secondary cone $\overline{\triangle(\mathcal{D})}$ then for arbitrary positive real numbers $\lambda_1,\ldots, \lambda_n$ the form $Q=\lambda_1Q_1+\ldots+\lambda_nQ_n$ defines:
\begin{enumerate}
\item the Delone tiling $D(Q)$ which is a common refinement of tilings $D(Q_1), \ldots, D(Q_n)$.
\item the Dirichlet-Voronoi polytope $V(Q)$ which is the Minkowski sum $\lambda_1V(Q_1)+\ldots+\lambda_nV(Q_n)$. 
\end{enumerate}
\end{thm}

That means that any parallelohedron can be represented as Minkowski sum of Dirichlet-Voronoi polytopes correspondent to rigid forms --- {\it mainstay parallelohedra} \cite{BolR}.

In dimensions $2$ and $3$ the only rigid forms are forms of the rank $1$. In that case Delone tilings formed by families of parallel hyperplanes and called {\it dicings} \cite{Erd} and correspondent Dirichlet-Voronoi polytopes are zonotopes.

\begin{defin}
A convex polytope $Z\subset \R^d$ is called {\it zonotope} if it can be represented as a Minkowski sum of finite number of segments. Equivalently, any zonotope $Z$ is a projection of some cube of dimension $n\geq d$.
\end{defin}

In $\R^4$ there is a rigid form $Q$ that is not a form of rank $1$. This form corresponds to the lattice $D_4$ \cite{d4}, i.e. if we take an affine transformation $\mathcal{A}$ such that $Q$ will transform into Euclidean metrics then the lattice $\Z^d$ will transform into $D_4$ lattice. The Dirichlet-Voronoi polytope for the lattice $D_4$ is the regular 24-cell \cite{24cell}. 

So in $\R^4$ there are two families of parallelohedra. The first family consist of zonotopes, and the second one consist of (affine images of) Minkowski sums of regular $24$-cell with segments. We can distinguish these families by considering rigid forms that are included in correspondent summands. We are going to study the first family in the section \ref{zon} and the second family in the section \ref{non-zon}.

\section{Homologies of $\pi$-surface via Venkov graph}

\begin{defin}
For every parallelohedron $P$ we can construct its {\it Venkov graph} $G_P$ \cite{Ord}. The vertices of $G_P$ are pairs of opposite facets of $P$ and two vertices of $G_P$ are connected with blue (red) edge if they are in one 6-belt (4-belt). 

Since we are interested only in primitive faces of codimension 2 then we will need only blue edges of Venkov graph. The graph $G_P'$ with the same set of vertices and only with blue edges is called {\it primitive Venkov graph}.

As it was shown in \cite{Ord} the primitive Venkov graph is connected if and only if $P$ can not be represented as direct product of two parallelohedra of smaller dimensions.
\end{defin}

\begin{defin}
Define the set $C_P$ of {\it gain cycles} of the graph $G_P$ consist of:
\begin{itemize}
\item half-belt cycles, i.e. triangles with edges that connects pairs of facets parallel to one primitive face of codimension 2;
\item trivially contractible cycles, i.e. cycles that represents closed paths on the $\pi$-surface around one $(d-3)$-face of $P$.
\end{itemize}
\end{defin}

For example the primitive Venkov graph for elongated dodecahedron is shown on the left part of the figure \ref{vengraph}. And the right part of this figure shows examples of half-belt cycle (the blue cycle) and trivially contractible cycle (the green cycle).

\begin{figure}[!ht]
\begin{center}
\includegraphics[scale=0.5]{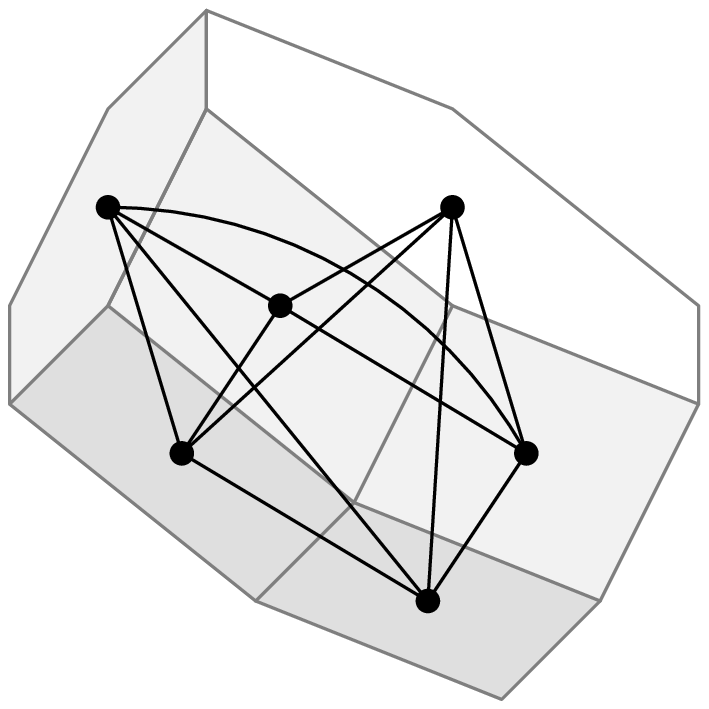}
\hskip 3cm
\includegraphics[scale=0.5]{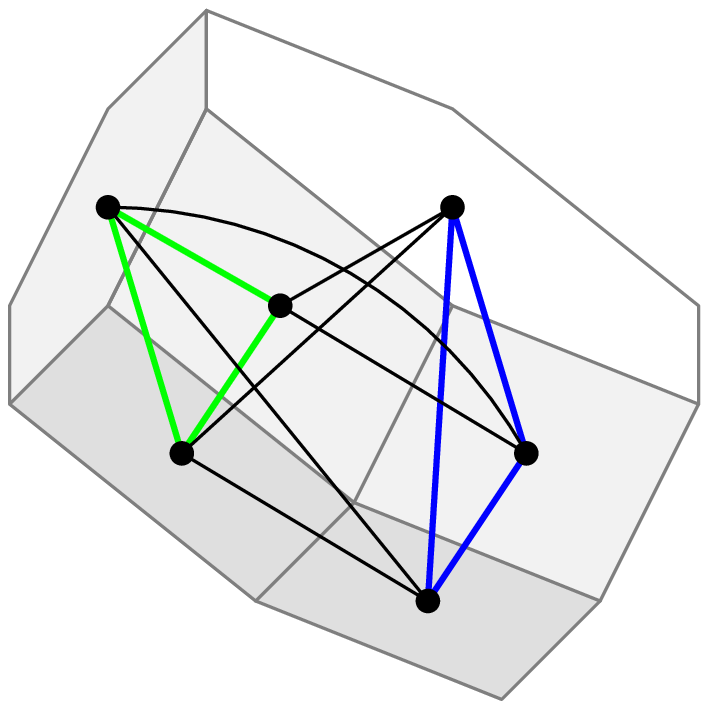}
\end{center}
\caption{The primitive Venkov graph of elongated dodecahedron and examples of gain cycles.}
\label{vengraph}
\end{figure}

\begin{thm}\label{graphcycles}
To show that theorem $\ref{homologies}$ is true for $P$ it is enough to check that Abelian group of cycles of the graph $G_P$ is generated by the set $C_P$ of gain cycles.
\end{thm}

\begin{proof}
It is easy to see that every cycle on $P_\pi$ can be represented as product of cycles in $G_P$ (we can just write down sequence of facets of $P$, i.e. vertices of $G_P$) and two cycles are homologically equivalent with rational coefficients if and only if they can be transformed into each other by products on trivially contractible cycles.

So the only thing we need to add that half-belt cycles on the $\pi$-surface of $P$ are represented with half-belt cycles of $G_P$.
\end{proof}

\section{Non-zonotopal four-dimensional parallelohedra}\label{non-zon}

In \cite{Val} it was shown that there are two non-equivalent full-dimensional secondary cones in $\mathcal{C}_4$ correspondent to non-zonotopal case with respect to transformations from the group $GL_4(\Z)$ and both can be obtained from rigid form $D_4$ by adding rank one forms. So to get the full picture of this case it is enough to take the Delone tiling of the lattice $D_4$ (we will denote it as $\mathcal{D}_{24}$) and describe all rank one forms that can be added to this tiling.

\begin{exam}
The $24$-cell is the convex hull of $24$ points: all $16$ points of the form $(\pm \frac12,\pm \frac12, \pm \frac12,\pm \frac12)$ and $8$ points with three coordinates equal to $0$ and one to $\pm 1$ (i.e. permutations of $(\pm 1,0,0,0)$).

This polytope is the Dirichlet-Voronoi polytopes for the lattice 
$$D_4=\bigl\{(x_1,x_2,x_3,x_4)\in\mathbb{Z}^d: x_1+x_2+x_3+x_4 \text{ is even}\bigr\}.$$
\end{exam}

\begin{exam}
The correspondent Delone tiling $\mathcal{D}_{24}$ consist three families of regular crosspolytopes $\mathcal{F}_1,\mathcal{F}_2$ and $\mathcal{F}_3$. The family $\mathcal{F}_1$ consist of crosspolytopes with centers $(1,0,0,0)+D_4$, i.e.
$$\mathcal{F}_1=\Bigl\{\conv\{(0,0,0,0),(2,0,0,0),(1,\pm 1,0,0),(1,0,\pm 1,0),(1,0,0,\pm 1)\}+D_4\Bigr\}.$$
In the same way two other families of crosspolytopes are the following:
\begin{multline}
\mathcal{F}_2=\Bigl\{\conv\{(0,0,0,0),(1,1,1,1),(1, 1,0,0),(1,0, 1,0),(1,0,0,1), (0, 1,1,0),\\ (0,1,0,1),(0,0,1,1)\}+D_4\Bigr\},\notag\end{multline}
and
\begin{multline}
\mathcal{F}_3=\Bigl\{\conv\{(0,0,0,0),(-1,1,1,1),(-1,1,0,0),(-1,0, 1,0),(-1,0,0,1), (0, 1,1,0),\\ (0,1,0,1),(0,0,1,1)\}+D_4\Bigr\}.\notag\end{multline}
\end{exam}

Now we will describe possible Delone tilings that can appear as refinements of the tiling $\mathcal{D}_{24}$. As we described it in the end of the section \ref{sec:class} the Voronoi polytopes for non-zonotopal case is an affine image of Minkowski sum of $24$-cell with several segments. So by theorem \ref{onecone} correspondent Delone tilings can be obtained be slicing the tiling $\mathcal{D}_{24}$ with several families of hyperplanes.

We can use the classification of Frank Vallentin \cite[Sect. 4.4.5]{Val} and obtain the following:

\begin{prop}
Each (non-zonotopal) Delone tiling can be obtained from the tiling into three families of crosspolytopes $\mathcal{F}_1,\mathcal{F}_2,\mathcal{F}_3$ with slicing it with several family of parallel hyperplanes without creating new vertices. These new hyperplanes satisfies the following properties:
\begin{enumerate}
\item Each family of hyperplanes slice only one family of crosspolytopes.
\item Each family of crosspolytopes can be sliced with upto three families of hyperplanes.
\end{enumerate}
\end{prop}
\begin{proof}
If we slice tiling $\mathcal{D}_{24}$ with a family of hyperplanes $\mathcal{H}$ then every crosspolytope should be sliced by its ``equator'' (i.e. three pairs of opposite vertices) or should not sliced at all. Due to symmetry with respect to symmetry group of $D_4$ it is enough to consider only one family of hyperplanes that slices one of crosspolytopes.

Without loss of generality we can assume that the crosspolytope 
$$C=\conv\{(0,0,0,0),(2,0,0,0),(1,\pm 1,0,0),(1,0,\pm 1,0),(1,0,0,\pm 1)\}\in\mathcal{F}_1$$
is sliced with the hyperplane $x_1=1$. Then the whole family $\mathcal{H}$ is the family $\{x_1=k,k\in \Z\}$. One can easily check that $\mathcal{H}$ slice every crosspolytope from the family $\mathcal{F}_1$ and slice no crosspolytopes from other families. So the first assertion of the proposition is done.

The second point is also quite clear. If we have more than three families of hyperplanes that slice $\mathcal{F}_i$ then we will get new vertices of the tiling in the centers of crosspolytopes from $\mathcal{F}_i$. And this is forbidden since the set of vertices of the tiling (i.e. lattice) should remain the same.
\end{proof}

\begin{thm}\label{nonzonthm}
If $P$ is a four-dimensional parallelohedron that contains affine image of $24$-cell as Minkowski summand (i.e. if correspondent Delone tiling is an affine image of a refinement of $\mathcal{D}_{24}$) then the theorem $\ref{homologies}$ is valid for $P$.
\end{thm}

\begin{proof}
First we describe what will happen with one crosspolytope $C$ if we slice it one, two or three times. One hyperplane cuts $C$ into two pyramids over octahedra (i.e. three-dimensional crosspolytopes). So the only other face added to Delone triagulation will be one three-dimensional crosspolytope (we will not describe additional four-dimensional tiles). 

The second hyperplane cuts each pyramid into doubly iterated pyramids over squares (i.e. pyramid over pyramid over square). So this construction adds one two-dimensional quadrangular face (bases of all four pyramids) and four three-dimensinal pyramidal faces (pyramids over squares) to the initial four-dimensional crosspolytope.

The third hyperplane cuts $C$ into eight simplices with common edge. The new faces of the tiling are: one edge, eight triangles, and twelve tetrahedra. So after the third cut of a single crosspolytope there will be again only tetrahedral three-dimensional faces.

If each family is cut one or three times, or left uncut then all two-dimensional faces of Delone tiling are triangular and correspondent Voronoi polytope $P$ satisfies the theorem of Zhitomirskii \ref{zhitthm}. In this case there are no non-primitive faces of codimension $2$ and $P_\pi\cong\mathbb{R}P^{3}$ and its rational homologies are trivial. Therefore the theorem \ref{homologies} holds for $P$.

Now consider that some families $\mathcal{F}_i$ with $i\in I\subset\{1,2,3\}$ were cut three times. Then the tiling $D'$ obtained from $D_{24}$ tiling by cutting only families $\mathcal{F}_i, i\in I$ satisfies the following two conditions:

\begin{itemize}
\item all three-dimensional faces of the tiling $D'$ are tetrahedral, since all $3$-faces of the initial tiling $D_{24}$ are tetrahedral and all additional $3$-faces are tetrahedral;
\item all two-dimensional faces are triangles, since all additional $2$-faces are triangles.
\end{itemize}

The second point implies that correspondent Voronoi polytope $P'$ satisfies conditions of the Zhitomirskii theorem \ref{zhitthm} and therefore its satisfies conditions of the theorem \ref{homologies}. So by lemma \ref{graphcycles} the group of cycles of the graph $G_{P'}$ is generated by the set $C_{P'}$.

Now we start to add hyperplanes that slice families from $\mathcal{F}_i,i\notin I$. This does not change vertices and edges of the graph $G_{P'}$ since there will be no additional edges (correspondent to facets of Voronoi polytope and vertices of the primitive Venkov graph) and triangles (correspondent to primitive 2-faces and edges of the primitive Venkov graph) in the Delone tiling. Therefore $G_P=G_{P'}$.

Now take a look on two sets of gain cycles $C_P$ and $C_{P'}$. The set of half-belt cycles of $P$ contains the set of half-belt cycles of $P'$ because no Delone triangles were deleted during transformation of the tiling $D'$ into tiling $D$. But the set of trivially contractible cycles could shrink during this transformation. How this can happen? If there is a 3-dimensional polytope $Q$ with vertex $O$ in the Delone tiling $D'$ and all 2-faces of $Q$ incident to $O$ are triangles then there is a trivially contractible cycle in the graph $G_{P'}$ correspondent to sequence of all edges of $Q$ incident to $O$. 

Therefore trivially contractible cycles could ``disappear'' only on the second slicing of some family if we slice an octahedron $OABO'A'B'$ by its equator $OAO'A'$. In that case trivially contractible cycle correspondent to edges $OA-OB-OA'-OB'$ will be destroyed. But this cycle was in the initial tiling $D'$ and it was not trivially contractible becuase in the tiling $D'$ all trivially contractible cycles are formed by three edges (all 3-faces of this tiling are tetrahedra). Since theorem \ref{graphcycles} is true for the tiling $D'$ (and moreover it is true that group of cycles of $G_{P'}$ was generated trivially contractible cycles in $C_{P'}$) and no cycle from $C_{P'}$ was destroyed then the theorem \ref{graphcycles} is true for $P$.
\end{proof}

\section{Zonotopal four-dimensional parallelohedra}\label{zon}

There are 17 four-dimensional space-filling zonotopes. We will list (almost) all ``interesting'' primitive Venkov graphs and corresponding gain cycles in the Appendix \ref{zonotopes}. For example we will not present primitive Venkov graphs for zonotopes that satisfying the case of Zhitomirskii (theorem \ref{zhitthm}) since the correspondent $\delta$-surface is the sphere $\mathbb{S}^3$ and the theorem \ref{homologies} is true. Also we will not present graphs for direct products $Z=Z_1\times Z_2$ where $Z_1$ and $Z_2$ are space-filling zonotopes of smaller dimensions since in that case $G_Z=G_{Z_1}\cup G_{Z_2}$ and $C_Z=C_{Z_1}\cup C_{Z_2}$ and the theorem \ref{graphcycles} is true due to induction.

Among these 17 zonotopes there is one zonotope correspondent to graphical zonotopal lattice $K_{3,3}$ (due to classification from \cite[Sect. 3.5]{Val}); this is the zonotope correspondent to dicing $\mathcal{D}^2(4)$ from \cite[Sect. 7]{Erd}. This polytope satisfies Zhitomirski theorem \ref{zhitthm}; indeed from \cite[Table 1]{Erd} or \cite[p. 61]{Val} one can find that this polytope has 30 facets and then Dolbilin's index theorem \cite{Dol} implies that it can not have non-primitive ridges.

All other 16 can be obtained as zonotopes correspondent to cographical zonotopal lattices (see \cite[Sect. 3.5]{Val}); equivalent description from \cite[Sect. 7]{Erd} is the following: they can be obtained by deleting zone vectors from other maximal dicing $\mathcal{D}^1(4)$. The zonotope correspondent to this dicing is {\it permutahedron}, i.e. convex hull of all 120 points in $\mathbb{R}^5$ with permuted coordinates $1,2,3,4,5$ (it is easy to see that all vertices of permutahedron lies in one four-dimensional plane). So, these zonotopes are $\Pi$-zonotopes due to definition from \cite{Gar}.

We will use approach to $\Pi$-zonotopes from the work \cite{Gar}. Each four-dimensional $\Pi$-zonotope $Z$ can be represented as connected graph $G$ with 5 vertices enumerated with numbers from 1 to 5. An edge with vertices $i$ and $j$ corresponds to a zone vector $\ib e_i -\ib e_j$ of the zonotope (all points and vectors are considered in $\mathbb{R}^5$ but they all lies in one four-dimensional plane). Here $\ib e_k$ denotes the $k$-th standard vector of $\mathbb{R}^5$.

The following properties can be easily checked. See also \cite{Gar} for some of these properties.
\begin{enumerate}
\item Pairs of facets of $Z$ correspondent to partitions of vertices of $G$ into two non-empty subsets that generates connected subgraphs.
\item Ridges lying in one belt of $Z$ orrespondent to partitions of vertices of $G$ into three non-empty subsets that generates connected subgraphs. Moreover, these ridges are primitive and belt contains 6 facet if and only if there is at least one edge between any pair of subsets.
\item Some ridges from one family belongs to a pair of opposite facets if and only if correspondent ridge-defining partition in three sets is a refinement of facet-defining partition into two sets.
\item Projection of $Z$ along its zone vector $\ib e_i -\ib e_j$ is a $\Pi$-zonotope with graph $G'$ obtained from $G$ by gluing vertices $i$ and $j$ together (contracting the edge $ij$).
\item ``Local'' structure of $Z$ at some edge defined by a zone vector $\ib e_i -\ib e_j$ is the same as the structure of the zonotope projected along this edge at correspondent vertex. It means that section of the tiling by copies of $Z$ transversal to this edge is the same as the structure of the tiling by projected zonotopes.
\item If primitive Venkov graph of projected zonotope is generated by the set of gain cycles (i.e. if it is a rhombic dodecahedron, elongated dodecahedron or truncated octahedron) then correspondent subgraph of $G$ is also generated by set of gain cycles of $Z$.
\end{enumerate}

The first three properties allows us to construct primitive Venkov graph for a $\Pi$-zonotope $Z$. Though they allow us to reconstruct only half-belt cycles but not trivially contractible, since they do not carry any information about local structure at faces of codimension 2. The last three properties fills this gap, moreover they allow us to use ``extended'' conditions that arise from properties that projection of $\Pi$-zonotope is a $\Pi$-zonotope again and that all three-dimensional space-filling zonotopes are $\Pi$-zonotopes. We need this extension (from one neighbourhood of face of codimension 3 to subgraph generated by projected zonotope) because it is not easy to extract particular face of one zonotope from graph structure but much easier to extract properties of family of equivalent faces.

Now we will show how this works for one case of zonotope that is not a direct sum of two zonotopes and is not satisfying Zhitomirski theorem \ref{zhitthm}.

\begin{exam}\label{piexample}
We will construct primitive Venkov graph for $\Pi$-zonotope $Z$ with the graph shown on the figure \ref{pict:zongraph} and check that its group of cycles is generated by gain cycles.

\begin{figure}[!ht]
\begin{center}
\includegraphics[scale=1]{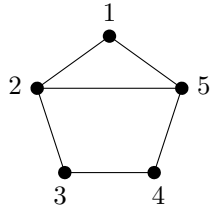}
\end{center}
\caption{Graph of a $\Pi$-zonotope.}
\label{pict:zongraph}
\end{figure}

This zonotope has six zone vectors $\ib e_1-\ib e_2$, $\ib e_2-\ib e_3$, $\ib e_3-\ib e_4$, $\ib e_4-\ib e_5$, $\ib e_5-\ib e_1$, and $\ib e_2-\ib e_5$. One can check that these vectors can not be divided into two mutually independent subset thus this zonotope can not be represented as direct sum of zonotopes of smaller dimension.

The following partitions defines pairs of facets of the zonotope $Z$.
\begin{enumerate}
\item[$F_1$:] $\{1\}$ and $\{2,3,4,5\}$; 
\item[$F_2$:] $\{1,2\}$ and $\{3,4,5\}$;
\item[$F_3$:] $\{1,5\}$ and $\{2,3,4\}$; 
\item[$F_4$:] $\{1,2,3\}$ and $\{4,5\}$;
\item[$F_5$:] $\{1,2,5\}$ and $\{3,4\}$; 

\item[$F_{6}$:] $\{1,4,5\}$ and $\{2,3\}$; 
\item[$F_{7}$:] $\{1,2,3,4\}$ and $\{5\}$; 
\item[$F_{8}$:] $\{1,2,3,5\}$ and $\{4\}$; 
\item[$F_{9}$:] $\{1,2,4,5\}$ and $\{3\}$; 
\item[$F_{10}$:] $\{1,3,4,5\}$ and $\{2\}$; 
\end{enumerate}
For example, partition $\{2,5\}$ and $\{1,3,4\}$ does not define a pair of facets since subgraph on vertices 1, 3, and 4 is not connected.

The following partitions defines belts of ridges of the zonotope $Z$.

\begin{enumerate}
\item[$f_1$:] $\{1\}$, $\{2\},$ and $\{3,4,5\}$. This is a primitive belt with facets $F_1$, $F_2$, and $F_{10}$.
\item[$f_2$:] $\{1\}$, $\{2,3\},$ and $\{4,5\}$. This is a primitive belt with facets $F_1$, $F_4$, and $F_6$.
\item[$f_3$:] $\{1\}$, $\{2,5\},$ and $\{3,4\}$. This is NOT a primitive belt with facets $F_1$ and $F_5$.
\item[$f_4$:] $\{1\}$, $\{2,3,4\}$, and $\{5\}$. This is a primitive belt with facets $F_1$, $F_3$, and $F_{7}$.
\item[$f_5$:] $\{1\}$, $\{2,3,5\}$, and $\{4\}$. This is NOT a primitive belt with facets $F_1$ and $F_8$.
\item[$f_6$:] $\{1\}$, $\{2,4,5\}$, and $\{3\}$. This is NOT a primitive belt with facets $F_1$ and $F_9$.
\item[$f_7$:] $\{1,2\}$, $\{3\},$ and $\{4,5\}$. This is a primitive belt with facets $F_2$, $F_4$, and $F_9$.
\item[$f_8$:] $\{1,2\}$, $\{3,4\},$ and $\{5\}$. This is a primitive belt with facets $F_2$, $F_5$, and $F_7$.
\item[$f_{9}$:] $\{1,5\}$, $\{2\},$ and $\{3,4\}$. This is a primitive belt with facets $F_3$, $F_5$, and $F_{10}$.
\item[$f_{10}$:] $\{1,5\}$, $\{2,3\},$ and $\{4\}$. This is a primitive belt with facets $F_3$, $F_6$, and $F_{8}$.

\item[$f_{11}$:] $\{1,2,3\}$, $\{4\},$ and $\{5\}$. This is a primitive belt with facets $F_4$, $F_7$, and $F_{8}$.
\item[$f_{12}$:] $\{1,2,5\}$, $\{3\},$ and $\{4\}$. This is a primitive belt with facets $F_5$, $F_8$, and $F_{9}$.
\item[$f_{13}$:] $\{1,4,5\}$, $\{2\},$ and $\{3\}$. This is a primitive belt with facets $F_6$, $F_9$, and $F_{10}$.

\end{enumerate}
So, the primitive Venkov graph $G_Z$ of $Z$ has 10 vertices, 30 edges (number of primitive belts times 3). This graph is shown on the figure \ref{pict:vengraph}.

\begin{figure}[!ht]
\begin{center}
\includegraphics[scale=1]{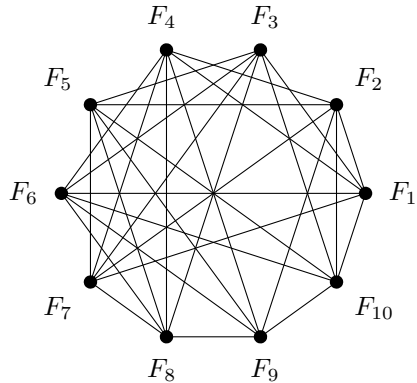}
\end{center}
\caption{Primitive Venkov graph $G_Z$ of the $\Pi$-zonotope $Z$.}
\label{pict:vengraph}
\end{figure}

Now we will reconstruct trivially contractible cycles. We will examine local structures at faces of codimension 3. In the case of four-dimensional polytopes we need to examine local structure at edges. We start from family of edges defined by zone vector $\ib e_4 - \ib e_5$. After projection along this edge we will get the $\Pi$-zonotope with the graph on the figure \ref{pict:projected}. To obtain this graph we glue together vertices $4$ and $5$ of the graph from the figure \ref{pict:zongraph} and remove loops and double edges.

\begin{figure}[!ht]
\begin{center}
\includegraphics[height=2cm]{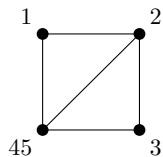}
\end{center}
\caption{Graph of the projected $\Pi$-zonotope.}
\label{pict:projected}
\end{figure}

This is the graph of elongated dodecahedron. So structure of tiling by copies of $Z$ at edges parallel to vector $\ib e_4-\ib e_5$ is (combinatorially, topologically) equivalent to the tiling by elongated dodecahedra at vertices. The theorem \ref{graphcycles} is true for elongated dodecahedron (one can check this easily as well as for all other three-dimensional parallelohedra) and therefore cycles of subgraph of the graph $G_Z$ generated by projection parallel to $\ib e_4 - \ib e_5$ are generated by gain cycles of the whole zonotope $Z$.

This subgraph has vertices $F_1$, $F_2$, $F_4$, $F_6$, $F_9$, and $F_{10}$ and shown on the figure \ref{pict:vensubgraph}. One can check that this subgraph is isomorphic to the primitive Venkov graph of elongated dodecahedron on the figure \ref{vengraph}.

\begin{figure}[!ht]
\begin{center}
\includegraphics[scale=1]{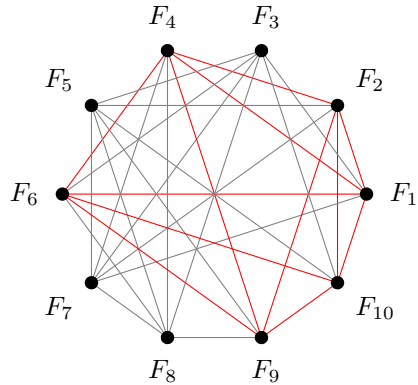}
\end{center}
\caption{Subgraph generated by projection parallel to zone vector $\ib e_4-\ib e_5$.}
\label{pict:vensubgraph}
\end{figure}

Similarly, we can construct projections along other zone vectors and fine additional subgraphs of the graph $G$ with cycles generated by gain cycles. These cycles are the following.
\begin{itemize}
\item Projection along $\ib e_1-\ib e_2$ is rhombic dodecahedron (see remark after this example on how we can distinguish three-dimensional projections); correspondent gain cycles generates cycles of the subgraph on facets $F_2$, $F_4$, $F_5$, $F_7$, $F_8$, and $F_9$.

\item Projection along $\ib e_2-\ib e_3$ is elongated dodecahedron; correspondent gain cycles generates cycles of the subgraph on facets $F_1$, $F_3$, $F_4$, $F_6$, $F_7$, and $F_8$.

\item Projection along $\ib e_3-\ib e_4$ is elongated dodecahedron; correspondent gain cycles generates cycles of the subgraph on facets $F_1$, $F_2$, $F_3$, $F_5$, $F_7$, and $F_{10}$.

\item Projection along $\ib e_5-\ib e_1$ is rhombic dodecahedron; correspondent gain cycles generates cycles of the subgraph on facets $F_3$, $F_5$, $F_6$, $F_8$, $F_9$, and $F_{10}$.

\item Projection along $\ib e_2-\ib e_5$ is hexagonal prism (existence of this projection shows that $Z$ is not the Ordine's case from theorem \ref{ordthm}); correspondent gain cycles generates cycles of the subgraph on facets $F_1$, $F_5$, $F_8$, and $F_9$.
\end{itemize}

Now one can check that cycles from listed subgraphs generates all cycles of the graph $G_Z$. For example we can do that by checking that subgroup generated by cycles from listed subgraphs has order exactly $21=30-10+1$ by finding $21$ independent cycles. This proves that $Z$ satisfies conditions of the theorem \ref{graphcycles}.
\end{exam}

\begin{rem}
There are six non-isomorphic connected graphs with four vertices. These graphs are shown on the figure \ref{pict:3dgraphs}. All these graphs can be obtained as graphs of projections of $\Pi$-zonotopes along zone vectors. 

\begin{figure}[!ht]
\begin{center}
\begin{subfigure}[b]{0.3\textwidth}
\begin{center}\includegraphics[width=2cm]{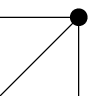}\end{center}
\caption{}
\label{pict:troc}
\end{subfigure}
\begin{subfigure}[b]{0.3\textwidth}
\begin{center}\includegraphics[width=2cm]{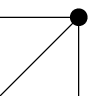}\end{center}
\caption{}\label{pict:eldodec}
\end{subfigure}
\begin{subfigure}[b]{0.3\textwidth}
\begin{center}\includegraphics[width=2cm]{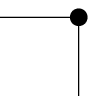}\end{center}
\caption{}\label{pict:rhdodec}
\end{subfigure}\\
\begin{subfigure}[b]{0.3\textwidth}
\begin{center}\includegraphics[width=2cm]{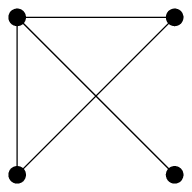}\end{center}
\caption{}
\label{pict:hexprism}
\end{subfigure}
\begin{subfigure}[b]{0.3\textwidth}
\begin{center}\includegraphics[width=2cm]{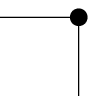}\end{center}
\caption{}\label{pict:cube1}
\end{subfigure}
\begin{subfigure}[b]{0.3\textwidth}
\begin{center}\includegraphics[width=2cm]{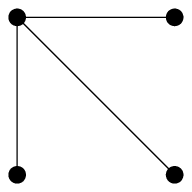}\end{center}
\caption{}\label{pict:cube2}
\end{subfigure}
\end{center}
\caption{Non-isomorphic connected graphs with four vertices.}
\label{pict:3dgraphs}
\end{figure}

The upper row of this figure represents irreducible parallelohedra, namely, truncated octahedron on the figure \ref{pict:troc}, elongated dodecahedron on the figure \ref{pict:eldodec}, and rhombic dodecahedron on the figure \ref{pict:rhdodec}. And the lower row illustrate graphs for reducible ones, hexagonal prism on the figure \ref{pict:hexprism}, and two represntations of a cube on figures \ref{pict:cube1} and \ref{pict:cube2}. Thus we can recognise any projection by its graph representation.
\end{rem}

\begin{thm}\label{zonthm}
Conditions of the theorem $\ref{graphcycles}$ is true for all four-dimensional space-filling zonotopes.
\end{thm}

\begin{proof}
The zonotope $K_{3,3}$ satisfies Zhitomirski theorem \ref{zhitthm} and therefore is satisfies conditions of the theorem \ref{graphcycles} too.

For all other $\Pi$-zonotopes (correspondent to cographical lattices) we can use alogrithm described in the example \ref{piexample}. Correspondent primitive Venkov graphs or additional arguments are shown in the appendix \ref{zonotopes}. This exhaustive method finishes the proof.
\end{proof}

\section{Conclusions and remarks}

Two theorems \ref{nonzonthm} and \ref{zonthm} give us the main result of this paper.

\begin{thm}
Any four-dimensional parallelohedron satisfies conditions of the theorem $\ref{homologies}$.
\end{thm}

This theorem does not give a new proof of the Voronoi conjecture for four-dimensional case. But it shows the theorem \ref{homologies} covers some additional cases of the Voronoi conjecture that had not been covered by Ordine's theorem \ref{ordthm}. One of these examples is shown in the example \ref{piexample}.

\appendix

\section{Primitive Venkov graphs of some four-dimensional $\Pi$-zonotopes}\label{zonotopes}

\newcounter{rownum}
\setcounter{rownum}{0}

\begin{longtable}{|c|c|l|c|}
\hline
\#&Graph&Facets or special properties&Primitive Venkov graph\\
\hline
\addtocounter{rownum}{1}\arabic{rownum} &\includegraphics[scale=1]{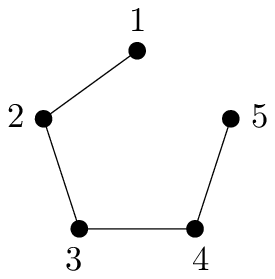}&\parbox[b]{0.35\textwidth}{This is a four-dimensional cube}&\\\hline
\addtocounter{rownum}{1}\arabic{rownum} &\includegraphics[scale=1]{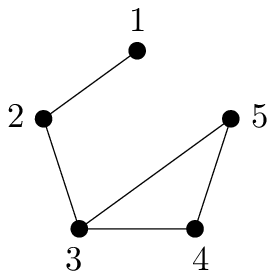}&\parbox[b]{0.35\textwidth}{This is a direct sum of a segment and hexagonal prism}&\\\hline
\addtocounter{rownum}{1}\arabic{rownum} &\includegraphics[scale=1]{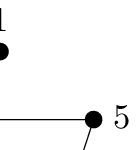}&\parbox[b]{0.35\textwidth}{This is a direct sum of a segment and rhombic dodecahedron}&\\\hline
\addtocounter{rownum}{1}\arabic{rownum} &\includegraphics[scale=1]{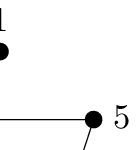}&\parbox[b]{0.35\textwidth}{This is a direct sum of a segment and elongated dodecahedron}&\\\hline
\addtocounter{rownum}{1}\arabic{rownum} &\includegraphics[scale=1]{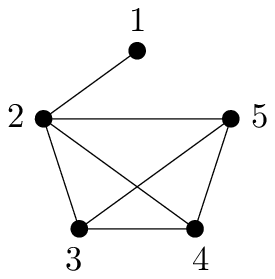}&\parbox[b]{0.35\textwidth}{This is a direct sum of a segment and truncated octahedron}&\\\hline
\addtocounter{rownum}{1}\arabic{rownum} &\includegraphics[scale=1]{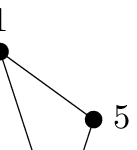}&\parbox[b]{0.35\textwidth}{This is a direct sum of two triangles in complementary planes}&\\\hline
\addtocounter{rownum}{1}\arabic{rownum} &\includegraphics[scale=1]{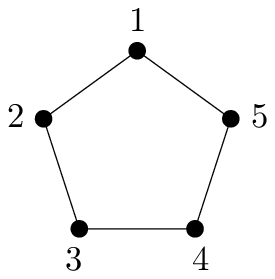}&\parbox[b]{0.35\textwidth}{This zonotope satisfies Zhitomirski theorem \ref{zhitthm} since all its projections along zone vectors are rhombic dodecahedra}&\\\hline
\addtocounter{rownum}{1}\arabic{rownum} &\includegraphics[scale=1]{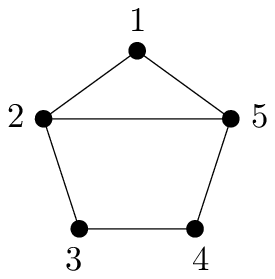}&\parbox[b]{0.35\textwidth}{
$F_1$: $\{1\}$ and $\{2,3,4,5\}$;\\
$F_2$: $\{1,2\}$ and $\{3,4,5\}$;\\
$F_3$: $\{1,5\}$ and $\{2,3,4\}$;\\ 
$F_4$: $\{1,2,3\}$ and $\{4,5\}$;\\
$F_5$: $\{1,2,5\}$ and $\{3,4\}$;\\ 
$F_{6}$: $\{1,4,5\}$ and $\{2,3\}$;\\ 
$F_{7}$: $\{1,2,3,4\}$ and $\{5\}$;\\ 
$F_{8}$: $\{1,2,3,5\}$ and $\{4\}$;\\ 
$F_{9}$: $\{1,2,4,5\}$ and $\{3\}$;\\ 
$F_{10}$: $\{1,3,4,5\}$ and $\{2\}$ 
}
&
\includegraphics[scale=1]{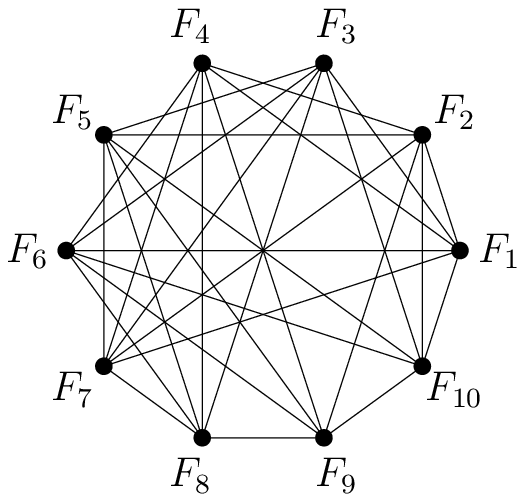}\\\hline
\addtocounter{rownum}{1}\arabic{rownum} &\includegraphics[scale=1]{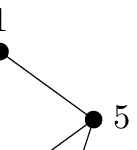}&\parbox[b]{0.35\textwidth}{
$F_1$: $\{1\}$ and $\{2,3,4,5\}$;\\
$F_2$: $\{1,2\}$ and $\{3,4,5\}$;\\
$F_3$: $\{1,5\}$ and $\{2,3,4\}$;\\ 
$F_4$: $\{1,2,3\}$ and $\{4,5\}$;\\
$F_5$: $\{1,2,4\}$ and $\{3,5\}$;\\ 
$F_{6}$: $\{1,3,5\}$ and $\{2,4\}$;\\ 
$F_{7}$: $\{1,4,5\}$ and $\{2,3\}$;\\ 
$F_{8}$: $\{1,2,3,4\}$ and $\{5\}$;\\ 
$F_{9}$: $\{1,2,3,5\}$ and $\{4\}$;\\ 
$F_{10}$: $\{1,2,4,5\}$ and $\{3\}$;\\ 
$F_{11}$: $\{1,3,4,5\}$ and $\{2\}$ 
}&\includegraphics[scale=1]{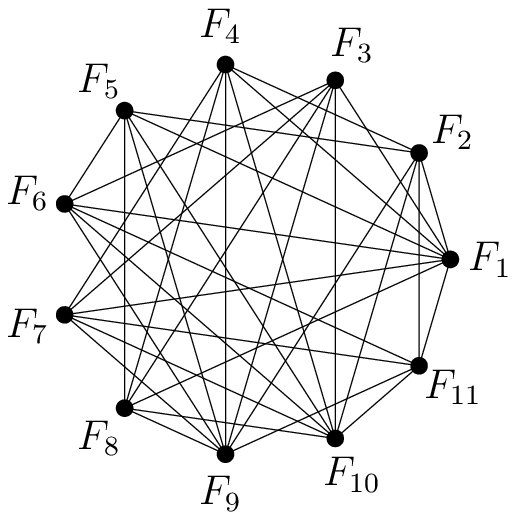}\\\hline

\addtocounter{rownum}{1}\arabic{rownum} &\includegraphics[scale=1]{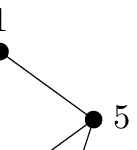}&\parbox[b]{0.35\textwidth}{
$F_1$: $\{1\}$ and $\{2,3,4,5\}$;\\
$F_2$: $\{1,2\}$ and $\{3,4,5\}$;\\
$F_3$: $\{1,5\}$ and $\{2,3,4\}$;\\ 
$F_4$: $\{1,2,3\}$ and $\{4,5\}$;\\
$F_5$: $\{1,2,4\}$ and $\{3,5\}$;\\
$F_6$: $\{1,2,5\}$ and $\{3,4\}$;\\ 
$F_{7}$: $\{1,3,5\}$ and $\{2,4\}$;\\ 
$F_{8}$: $\{1,4,5\}$ and $\{2,3\}$;\\ 
$F_{9}$: $\{1,2,3,4\}$ and $\{5\}$;\\ 
$F_{10}$: $\{1,2,3,5\}$ and $\{4\}$;\\ 
$F_{11}$: $\{1,2,4,5\}$ and $\{3\}$;\\ 
$F_{12}$: $\{1,3,4,5\}$ and $\{2\}$ 
}&\includegraphics[scale=1]{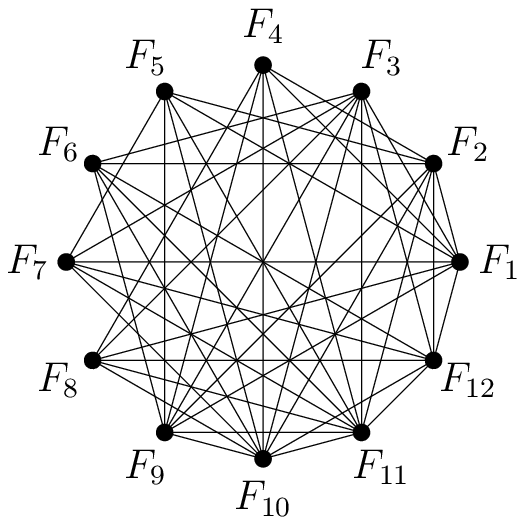}\\\hline

\addtocounter{rownum}{1}\arabic{rownum} &\includegraphics[scale=1]{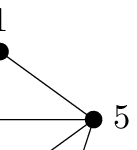}&\parbox[b]{0.35\textwidth}{
$F_1$: $\{1\}$ and $\{2,3,4,5\}$;\\
$F_2$: $\{1,2\}$ and $\{3,4,5\}$;\\
$F_3$: $\{1,5\}$ and $\{2,3,4\}$;\\ 
$F_4$: $\{1,2,3\}$ and $\{4,5\}$;\\
$F_5$: $\{1,2,4\}$ and $\{3,5\}$;\\ 
$F_{6}$: $\{1,3,5\}$ and $\{2,4\}$;\\ 
$F_{7}$: $\{1,4,5\}$ and $\{2,3\}$;\\ 
$F_{8}$: $\{1,2,3,4\}$ and $\{5\}$;\\ 
$F_{9}$: $\{1,2,3,5\}$ and $\{4\}$;\\ 
$F_{10}$: $\{1,2,4,5\}$ and $\{3\}$;\\ 
$F_{11}$: $\{1,3,4,5\}$ and $\{2\}$ 
}&\includegraphics[scale=1]{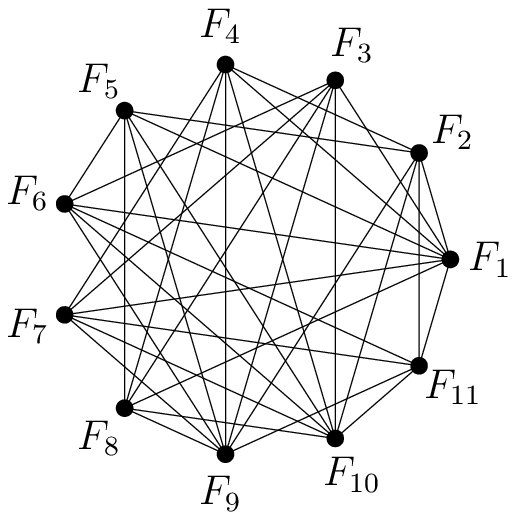}\\\hline
\addtocounter{rownum}{1}\arabic{rownum} &\includegraphics[scale=1]{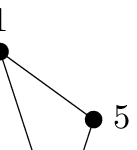}&\parbox[b]{0.35\textwidth}{
$F_1$: $\{1\}$ and $\{2,3,4,5\}$;\\
$F_2$: $\{1,2\}$ and $\{3,4,5\}$;\\
$F_3$: $\{1,5\}$ and $\{2,3,4\}$;\\ 
$F_4$: $\{1,2,3\}$ and $\{4,5\}$;\\
$F_5$: $\{1,2,5\}$ and $\{3,4\}$;\\ 
$F_{6}$: $\{1,4,5\}$ and $\{2,3\}$;\\ 
$F_{7}$: $\{1,2,3,4\}$ and $\{5\}$;\\ 
$F_{8}$: $\{1,2,3,5\}$ and $\{4\}$;\\ 
$F_{9}$: $\{1,2,4,5\}$ and $\{3\}$;\\ 
$F_{10}$: $\{1,3,4,5\}$ and $\{2\}$ 
}
&
\includegraphics[scale=1]{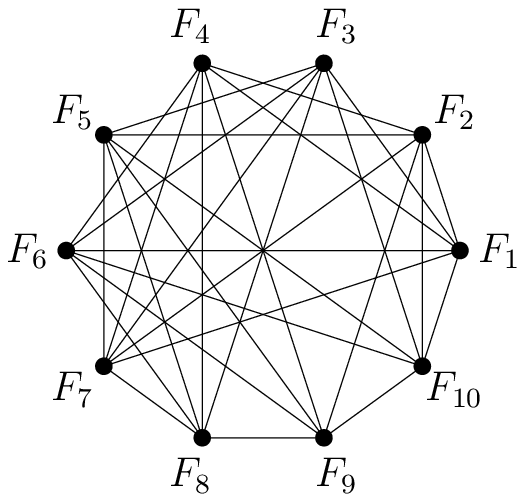}\\\hline
\addtocounter{rownum}{1}\arabic{rownum} &\includegraphics[scale=1]{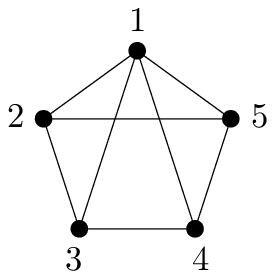}&\parbox[b]{0.35\textwidth}{
$F_1$: $\{1\}$ and $\{2,3,4,5\}$;\\
$F_2$: $\{1,2\}$ and $\{3,4,5\}$;\\
$F_3$: $\{1,3\}$ and $\{2,4,5\}$;\\
$F_4$: $\{1,4\}$ and $\{2,3,5\}$;\\
$F_5$: $\{1,5\}$ and $\{2,3,4\}$;\\ 
$F_6$: $\{1,2,3\}$ and $\{4,5\}$;\\
$F_7$: $\{1,2,5\}$ and $\{3,4\}$;\\ 
$F_8$: $\{1,3,4\}$ and $\{2,5\}$;\\ 
$F_{9}$: $\{1,4,5\}$ and $\{2,3\}$;\\ 
$F_{10}$: $\{1,2,3,4\}$ and $\{5\}$;\\ 
$F_{11}$: $\{1,2,3,5\}$ and $\{4\}$;\\ 
$F_{12}$: $\{1,2,4,5\}$ and $\{3\}$;\\ 
$F_{13}$: $\{1,3,4,5\}$ and $\{2\}$ 
}&\includegraphics[scale=1]{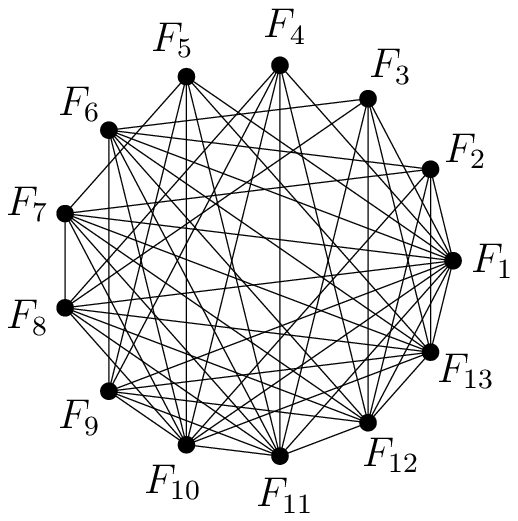}\\\hline

\addtocounter{rownum}{1}\arabic{rownum} &\includegraphics[scale=1]{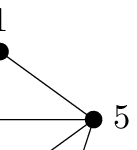}&\parbox[b]{0.35\textwidth}{
$F_1$: $\{1\}$ and $\{2,3,4,5\}$;\\
$F_2$: $\{1,2\}$ and $\{3,4,5\}$;\\
$F_3$: $\{1,5\}$ and $\{2,3,4\}$;\\ 
$F_4$: $\{1,2,3\}$ and $\{4,5\}$;\\
$F_5$: $\{1,2,4\}$ and $\{3,5\}$;\\
$F_6$: $\{1,2,5\}$ and $\{3,4\}$;\\ 
$F_{7}$: $\{1,3,5\}$ and $\{2,4\}$;\\ 
$F_{8}$: $\{1,4,5\}$ and $\{2,3\}$;\\ 
$F_{9}$: $\{1,2,3,4\}$ and $\{5\}$;\\ 
$F_{10}$: $\{1,2,3,5\}$ and $\{4\}$;\\ 
$F_{11}$: $\{1,2,4,5\}$ and $\{3\}$;\\ 
$F_{12}$: $\{1,3,4,5\}$ and $\{2\}$ 
}&\includegraphics[scale=1]{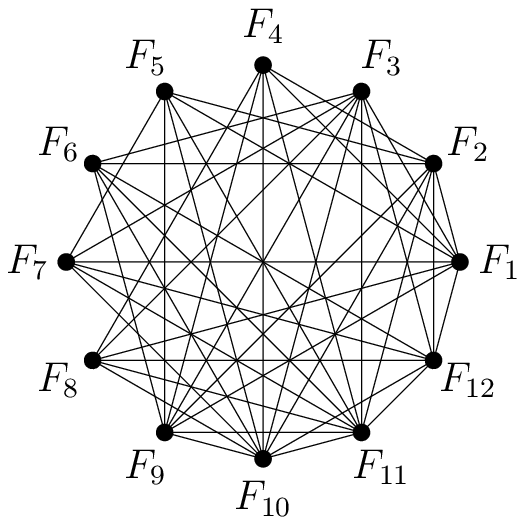}\\\hline
\addtocounter{rownum}{1}\arabic{rownum} &\includegraphics[scale=1]{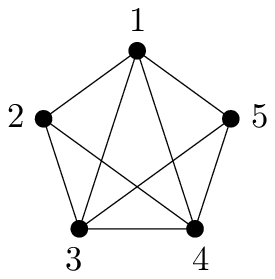}&\parbox[b]{0.35\textwidth}{
$F_1$: $\{1\}$ and $\{2,3,4,5\}$;\\
$F_2$: $\{1,2\}$ and $\{3,4,5\}$;\\
$F_3$: $\{1,3\}$ and $\{2,4,5\}$;\\
$F_4$: $\{1,4\}$ and $\{2,3,5\}$;\\
$F_5$: $\{1,5\}$ and $\{2,3,4\}$;\\ 
$F_6$: $\{1,2,3\}$ and $\{4,5\}$;\\
$F_7$: $\{1,2,4\}$ and $\{3,5\}$;\\
$F_8$: $\{1,2,5\}$ and $\{3,4\}$;\\ 
$F_9$: $\{1,3,5\}$ and $\{2,4\}$;\\ 
$F_{10}$: $\{1,4,5\}$ and $\{2,3\}$;\\ 
$F_{11}$: $\{1,2,3,4\}$ and $\{5\}$;\\ 
$F_{12}$: $\{1,2,3,5\}$ and $\{4\}$;\\ 
$F_{13}$: $\{1,2,4,5\}$ and $\{3\}$;\\ 
$F_{14}$: $\{1,3,4,5\}$ and $\{2\}$ 
}&\includegraphics[scale=1]{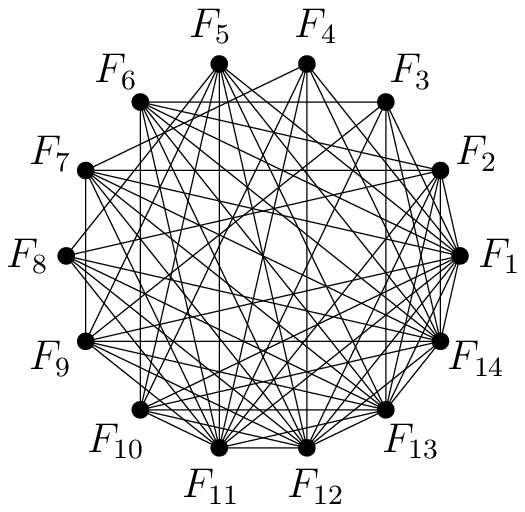}\\\hline
\addtocounter{rownum}{1}\arabic{rownum} &\includegraphics[scale=1]{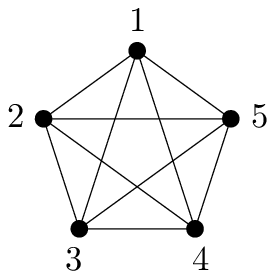}&\parbox[b]{0.35\textwidth}{This is four-dimensional permutahedron $\Pi_4$. It satisfies conditions of the Voronoi theorem \ref{vorthm}}&\\\hline

\end{longtable}

\end{document}